\documentclass[a4paper,11pt]{amsart}
\usepackage{a4wide}
\usepackage{enumerate}
\usepackage[utf8]{inputenc}
\usepackage{url}
\usepackage{amssymb}
\usepackage{hyperref}
\IfFileExists{\jobname.TEX}{}{\usepackage{hyperref,srcltx}\usepackage[utf8]{inputenc}}
\usepackage{array}

\numberwithin{equation}{section}
\renewcommand{\theequation}{\arabic{section}.\arabic{equation}}

\newcommand{\field}[1]{\mathbb{{#1}}}

\newcommand{\R}{\field{R}}

\newcommand{\Z}{\field{{Z}}}
\newcommand{\Q}{\field{{Q}}}
\newcommand{\K}{\field{{K}}}
\newcommand{\intpart}[1]{\left\lfloor#1\right\rfloor}
\newcommand{\fracpart}[1]{\left\{#1\right\}}

\newcommand{\disc}{\Delta}
\renewcommand{\d}{\,\mathrm{d}}
\newcommand{\Imm}{\mathrm{Im}}
\newcommand{\Ree}{\mathrm{Re}}
\newcommand{\PP}{\mathfrak{p}}

\newcommand{\jbar}{\bar{\jmath}}

\newcommand{\Norm}{\textrm{\upshape N}}
\DeclareMathOperator{\asinh}{arcsinh}

\newtheorem{theorem}{Theorem}[section]
\newtheorem{lemma}[theorem]{Lemma}
\newtheorem*{lemma*}{Lemma}

\newtheorem{corollary}[theorem]{Corollary}
\newtheorem*{corollary*}{Corollary}

\theoremstyle{remark}

\newtheorem*{remark*}{Remark}
\newtheorem*{acknowledgements}{Acknowledgements}

\allowdisplaybreaks[4]

\begin{document}
\title[Zeros for Dedekind zeta functions]{Zeros of Dedekind zeta functions under GRH}

\author[L.~Greni\'{e}]{Lo\"{i}c Greni\'{e}}
\address[L.~Greni\'{e}]{Dipartimento di Ingegneria Gestionale, dell'Informazione e della Produzione\\
         Universit\`{a} di Bergamo\\
         viale Marconi 5\\
         24044 Dalmine (BG)
         Italy}
\email{loic.grenie@gmail.com}

\author[G.~Molteni]{Giuseppe Molteni}
\address[G.~Molteni]{Dipartimento di Matematica\\
         Universit\`{a} di Milano\\
         via Saldini 50\\
         20133 Milano\\
         Italy}
\email{giuseppe.molteni1@unimi.it}

\keywords{} \subjclass[2010]{Primary 11R42}


\begin{abstract}
Assuming GRH, we prove an explicit upper bound for the number of zeros of a Dedekind zeta
function having imaginary part in $[T-a,T+a]$. We also prove a bound for the multiplicity of the zeros.
\end{abstract}

\maketitle

\begin{center}
To appear in Math. Comp. 2015.
\end{center}

\section{Introduction and results}\label{sec:C1}
Let $\K$ be a number field of dimension $n_\K$ and let $\disc_\K$ be the absolute value of its
discriminant. Let $n_\K(T;a)$ denote the number of zeros $\rho = \beta + i\gamma$ of the Dedekind zeta
function $\zeta_\K$ with $|\gamma-T| \leq a$ and which are nontrivial, i.e. with $0 < \beta < 1$.\\
An upper bound can be deduced via the equality $n_\K(T;a)= \tfrac{1}{2}(N_\K(T+a)-N_\K(T-a^-))$ where
$N_\K(T)$ counts the nontrivial zeros with imaginary part in $[-T,T]$. In this way from the explicit bound
for $N_\K(T)$ recently proved by Trudgian~\cite{TrudgianIII} it follows that
\begin{equation}\label{eq:C1}
n_\K(T;a) \leq \Big(\frac{a}{\pi} + 0.230\Big)\big(\log\disc_\K + n_\K \log T\big)  + \text{lower order terms}
\qquad
\forall T>a+1,
\end{equation}
where the remaining terms are explicit, have lower order as a function of $T$, and can be estimated
independently of the discriminant. The constant $0.230$ in~\eqref{eq:C1} is the parameter $D_1$ in
\cite[Table~2]{TrudgianIII} for $\eta=10^{-3}$ and is nearly the infimum $1/(\pi\log 4)$ for $D_1$.\\
The remainder term in the formula for $N_\K(T)$ is explicit but has the classical size $O(\log T)$.
This changes if one assumes the Generalized Riemann Hypothesis, since the works of
Littlewood~\cite{Littlewood2}, Selberg~\cite{Selberg1} and Lang~\cite{Lang8} show that in this case the
remainder term drops to $O(\tfrac{\log T}{\log\log T})$, so that now one gets
\begin{equation}\label{eq:C2}
n_\K(T;a) = \frac{a}{\pi}\big(\log\disc_\K + n_\K \log T\big)  + \text{lower order terms}.
\end{equation}
Recent computations of Carneiro, Chandee and Milinovich~\cite{CarneiroChandeeMilinovich} show that for
the Riemann zeta function the constant in the remainder term of $N_\Q(T)$ is
$(\tfrac{1}{4}+o(1))\tfrac{\log T}{\log\log T}$, at most. Applying Lang's heuristic~\cite{Lang8}, the
general case should be similar to $(1+o(1))\frac{\log\disc_\K +n_\K \log T}{\log\log T}$ and thus the
lower order terms in~\eqref{eq:C1} are probably such that
\[
n_\K(T;a) \leq \Big(\frac{a}{\pi} + \frac{c}{\log\log T}\Big)\big(\log\disc_\K + n_\K \log T\big)  + \text{lower order terms}
\]
with an absolute constant $c\lesssim 1$. Due to the very slow decay of the function $1/\log\log T$, this
tentative formula would improve on a result of type~\eqref{eq:C1} only for very large $T$. As a
consequence, for numerical applications it is interesting to work out a totally explicit bound for
$n_\K(T;a)$ under GRH, with an asymptotically non-optimal but small constant in front of the main term
$\log\disc_\K + n_\K \log T$, and possibly small constants in every other position. In this spirit, in
this paper we prove the following results: the first one for the zeros in the window $[T-a,T+a]$, the
second for the multiplicity of a zero.
\begin{theorem}\label{th:C1}
Assume GRH. Then one has
\begin{align}
n_\K(T;a)
\leq& \frac{a}{2}\tilde{f}_\K\Big(\frac{1}{2} + \frac{a}{4} + iT\Big)
&& \forall a\in(0,2),
\quad      T\geq 10+a                                                                \label{eq:C3}
\intertext{and}
n_\K(T;0^+)
\leq& \frac{3}{10}(2\sigma-1)\tilde{f}_\K(\sigma+iT)
&& \forall \sigma\in\Big(\frac{1}{2},1\Big),
\quad      T\geq 10,                                                                 \label{eq:C4}
\end{align}
where
\[
\tilde{f}_\K(\sigma+iT)
:= Q
  + 2\Big(\frac{n_\K}{1-\sigma}
          + \frac{\log(\frac{1}{2\sigma-1})}{\pi}
          + \frac{0.64}{2\sigma-1}
          + 1.37
     \Big)Q^{2-2\sigma}
  + \Big(\frac{0.14}{2\sigma-1}
         - 20\Big) n_\K
\]
and $Q:= \log\disc_\K + (\log T + 20)n_\K + 11$.
\end{theorem}
The main term of the bound~\eqref{eq:C3} is $\frac{a}{2}(\log\disc_\K + n_\K\log T)$. It improves
on the bound~\eqref{eq:C1} for $n_\K(T;a)$ whenever $a\leq 1.265\ldots$.
%
%
For example we get
\begin{align*}
n_\K\Big(T;\frac{1}{2}\Big)
&\leq \frac{1}{4} Q
        + (1.4 n_\K + 2.2) Q^{3/4}
        - 4 n_\K
&& \forall T\geq 10.5
%
%
\intertext{and}
n_\K(T;1)
&\leq \frac{1}{2} Q
        + (4 n_\K + 2.9) \sqrt{Q}
        - 9 n_\K
&& \forall T\geq 11.
\end{align*}
\medskip

The bound~\eqref{eq:C4} for the multiplicity is stronger than what we can deduce from~\eqref{eq:C3} in
the limit $a\to 0^+$ (to compare the results take $a=(4\sigma-2)$ in~\eqref{eq:C3}). Moreover, every
$\sigma < 0.883$ in~\eqref{eq:C4} improves on what one can deduce from~\eqref{eq:C1} in the limit $a\to
0^+$; for example for $\sigma=3/4$ we get
\begin{align*}
n_\K(T;0^+)
\leq
\frac{3}{20}Q
  + (1.2n_\K + 0.9)\sqrt{Q}
  - 2.9 n_\K
\qquad \forall T\geq 10.
\end{align*}
%
%
However, for a better result the form of~\eqref{eq:C4} suggests to try with a $\sigma$ such that
$2\sigma-1\to 0$. In fact, a proper choice of $\sigma$ proves the following claim.
\begin{corollary}\label{cor:C1}
Assume GRH. Let $Q$ as in Theorem~\ref{th:C1} and let $L:=\log Q$. Suppose $T\geq 10$, then
\[
n_\K(T;0^+)
\leq \Big(0.3\log L
          + 0.4
          + 0.2\frac{\log^2 L}{L}
          + \frac{\log L}{L}(1.9 n_\K + 0.9)
     \Big)\frac{Q}{\log Q}.
\]
\end{corollary}
%
\noindent%
\begin{proof}
Let $\epsilon := 2\sigma -1$. From~\eqref{eq:C4} we get
\begin{align*}
\frac{10}{3Q}n_\K(T;0^+)
\leq& \epsilon
     + \Big(\frac{4\epsilon}{1-\epsilon} n_\K + \frac{2}{\pi}\epsilon|\log \epsilon| + 2.76\epsilon + 1.28\Big)Q^{-\epsilon}
     + \frac{0.14n_\K}{Q}
\end{align*}
and setting $\epsilon = \frac{\log L}{\log Q}$ (and using $T\geq 10 \implies Q\geq 33
\implies \epsilon\leq 0.36$) we get
\begin{align*}
\frac{10}{3Q}n_\K(T;0^+)
\leq& \frac{\log L}{\log Q}
      + \Big(6.3\frac{\log L}{\log Q}n_\K
             + 2\frac{\log^2 L}{\pi\log Q}
             + 2.76\frac{\log L}{\log Q}
             + 1.28
        \Big)\frac{1}{\log Q}.
\end{align*}
%
\end{proof}
\noindent%
As recalled before, we already know that under GRH the multiplicity of $\tfrac{1}{2}+iT$ is
$O\big(\tfrac{\log T}{\log\log T}\big)$, thus Corollary~\ref{cor:C1} is weaker than the best known
result, but only by the presence of an extra $\log L$, i.e. a triple log in $T$, in the numerator.
Moreover, it is uniform in $\K$, and totally explicit.
\smallskip\\
Due to the presence of the extra factor $\log L$, every explicit bound of the form $n_\K(T;0^+) \leq c
\frac{\log T}{\log\log T}$ based on Corollary~\ref{cor:C1} holds only for a bounded range; nevertheless,
the following result shows that for the Riemann zeta function this range is extremely large, even for a
small value of $c$.
\begin{corollary}\label{cor:C2}
Assume RH and $\log T\leq 10^{70593}$. Then
\[
n_\Q(T;0^+)
\leq \frac{4\log T}{\log \log T}.
\]
\end{corollary}
\begin{proof}
By Corollary~\ref{cor:C1} it is sufficient to prove that
\begin{equation}\label{eq:C5}
\Big(\frac{3\log L}{10} + 2.8\frac{\log L}{L} + 0.2\frac{\log^2L}{L} + 0.4\Big)\frac{Q}{\log Q}
\leq 4 \frac{\log T}{\log\log T}
\end{equation}
with $Q= \log T + 31$ and $L=\log Q$. Taking account of the fact that the zeros $\rho$ of the Riemann
zeta function with $|\Imm(\rho)| < 10^{10}$ are simple~\cite{Bober} (actually this has been verified
up to $10^{12}$ \cite{Gourdon}, but we prefer to base our result on a doubly checked computation), we may
assume $\log T\geq 23$. In terms of $T$, Inequality~\eqref{eq:C5} is difficult, however we can verify that
both
\[
\frac{3}{10}\log L + 2.8\frac{\log L}{L} + 0.2\frac{\log^2L}{L} + 0.4
\leq 2
\]
and
\[
\frac{Q}{\log Q}
\leq 2\frac{\log T}{\log\log T}
\]
hold for $23\leq\log T\leq 10^{55}$ (the first one as a function of $L$, the second one as a function of
$\log T$).
%
Thus we can assume $\log T\geq 10^{55}$. Under this hypothesis~\eqref{eq:C5} is implied by
\[
\frac{3}{10}\log L + 2.8\frac{\log L}{L} + 0.2\frac{\log^2L}{L} + 0.4
\leq 4(1-10^{-10})
\]
which holds for $L \leq 162546.6$.
%
\end{proof}

Theorem~\ref{th:C1} is proved in two steps, following an idea which we have introduced
in~\cite{GrenieMolteni1}: let $f_\K(s) := \sum_{\rho} \Ree\big(\frac{2}{s-\rho}\big)$, where the sum is
on the set of nontrivial zeros of $\zeta_\K$. First we exploit the fact that the terms appearing in the
sum defining $f_\K$ are all positive and depend on the zeros of $\zeta_\K$ to find a suitable combination
of values of $f_\K$ providing an upper bound for $n_\K(T;a)$ and $n_\K(T;0^+)$; then we bound $f_\K(s)$
with $\tilde{f}_\K(s)$ in the critical strip. Both steps depend on GRH.
\smallskip

To bound $f_\K$ we will use a preliminary explicit upper bound for $n_\K(T;1)$ which we deduced from a
crude version of~\eqref{eq:C1}. A virtuous circle appears here because the argument could be iterated
producing better and better bounds. However, preliminary considerations suggest that the improvement is
quite marginal and affects only the secondary constants.


\begin{acknowledgements}
A special thank to Alberto Perelli for valuable remarks and comments.
We would like to expand our thanks to the reviewer for her/his constructive comments and suggestions.
\end{acknowledgements}

\section{Preliminary computations and the upper bound for $n_\K(T;a)$}\label{sec:C2}
For $\Ree(s) > 1$ we have
\[
-\frac{\zeta_\K'}{\zeta_\K}(s)
= \sum_{n=1}^{\infty} \tilde{\Lambda}_\K(n)n^{-s}
\quad
\text{with}
\quad
\tilde{\Lambda}_\K(n)
=
\begin{cases}
\displaystyle
\sum_{\PP|p,\, f_\PP|k} \log \Norm\PP & \text{if $n=p^k$},\\
\quad\ 0                              & \text{otherwise},
\end{cases}
\]
where $p$ is a prime number, $\PP$ a prime ideal in $\K$ above $p$, $\Norm\PP$ its absolute norm and
$f_\PP$ its residual degree. The formula for $\tilde{\Lambda}_\K$ shows that
$\tilde{\Lambda}_\K(n)\leq n_\K \Lambda(n)$ for every integer $n$, so that
\begin{equation}\label{eq:C6}
\Big|\frac{\zeta_\K'}{\zeta_\K}(s)\Big|\leq -n_\K\frac{\zeta'}{\zeta}(\sigma)
\qquad
\forall\sigma=\Ree(s)>1.
\end{equation}
The functional equation for $\zeta_\K$ reads
\begin{equation}\label{eq:C7}
\xi_\K(1-s) = \xi_\K(s)
\end{equation}
where
\begin{equation}\label{eq:C8}
\xi_\K(s) := s(s-1) \disc_\K^{s/2}\Gamma_\K(s)\zeta_\K(s)
\end{equation}
and, with $r_1$ and $r_2$ the number of real and complex embeddings of $\K$,
\begin{equation}\label{eq:C9}
\Gamma_\K(s) := \Big[\pi^{-\frac{s+1}{2}}\Gamma\Big(\frac{s+1}{2}\Big)\Big]^{r_2}
                \Big[\pi^{-\frac{s}{2}}  \Gamma\Big(\frac{s}{2}\Big)  \Big]^{r_1+r_2}.
\end{equation}
Since $\xi_\K(s)$ is an entire function of order $1$ and does not vanish at $s = 0$, one has
\begin{equation}\label{eq:C10}
\xi_\K(s) = e^{A_\K+B_\K s} \prod_\rho \Big(1 - \frac{s}{\rho}\Big) e^{s/\rho}
\end{equation}
for some constants $A_\K$ and $B_\K$, where $\rho$ runs through all the zeros of $\xi_\K(s)$, which are
precisely those zeros $\rho = \beta + i\gamma$ of $\zeta_\K$ for which $0 < \beta < 1$. We recall that the
zeros are symmetric with respect to the real axis, as a consequence of the fact that $\zeta_\K(s)$ is real
for $s\in\R$.\\
Differentiating~\eqref{eq:C8} and~\eqref{eq:C10} logarithmically we obtain the identity
\begin{equation}\label{eq:C11}
\frac{\zeta'_\K}{\zeta_\K}(s)
  = B_\K
   + \sum_\rho \Big(\frac{1}{s-\rho}+\frac{1}{\rho}\Big)
   - \frac{1}{2}\log \disc_\K
   - \Big(\frac{1}{s}+\frac{1}{s-1}\Big)
   - \frac{\Gamma'_\K}{\Gamma_\K}(s),
\end{equation}
valid identically in the complex variable $s$.\\
Stark~\cite[Lemma~1]{Stark1} proved that the functional equation~\eqref{eq:C7} implies that $B_\K =
-\sum_\rho \Ree(\rho^{-1})$, and that once this information is available one can use~\eqref{eq:C11} and
the definition of the gamma factor in~\eqref{eq:C9} to prove that the function $f_\K(s) :=
\sum_\rho\Ree\big(\frac{2}{s-\rho}\big)$ can be exactly computed via the alternative representation
\begin{equation}\label{eq:C12}
f_\K(s)
= 2\Ree\frac{\zeta'_\K}{\zeta_\K}(s)
    +\log\frac{\disc_\K}{\pi^{n_\K}}
    +\Ree\Big(\frac{2}{s}
             +\frac{2}{s-1}
        \Big)
    +(r_1+r_2)\Ree\frac{\Gamma'}{\Gamma}\Big(\frac{s}{2}\Big)
    +r_2\Ree\frac{\Gamma'}{\Gamma}\Big(\frac{s+1}{2}\Big).
\end{equation}
The relevance of this function for our problem comes from two facts: it is a sum on zeros each one
appearing with the weight $\Ree\big(\frac{2}{s-\rho}\big)$ which is \emph{positive} under GRH whenever
$\Ree(s)> \tfrac{1}{2}$, and it can be computed via the alternative formula~\eqref{eq:C12} which
\emph{does not involve} the zeros. For example, assuming GRH we get
\[
n_\K(T;a)
\leq \frac{1}{c(\sigma)}\sum_{\rho}\Ree\Big(\frac{2}{\sigma+iT-\rho}\Big)
=    \frac{1}{c(\sigma)}f_\K(\sigma+iT)
\qquad
T > a
\]
with $c(\sigma):=\tfrac{2\sigma-1}{(\sigma-1/2)^2+a^2}$, which is a lower bound for the weight of the
zeros counted by $n_\K(T;a)$. By~\eqref{eq:C12} the part depending on the discriminant in
$c(\sigma)^{-1}f_\K(\sigma+iT)$ is simply $c(\sigma)^{-1}\log\disc_\K$, hence to bound the contribution
of this parameter to $n_\K(T;a)$ we need to choose $\sigma$ such that $c(\sigma)$ is maximum. This
happens when $\sigma= \tfrac{1}{2} + a$, giving the bound
\begin{equation}\label{eq:C13}
n_\K(T;a) \leq a f_\K\Big(\frac{1}{2}+a+iT\Big)
\qquad
\forall T>a.
\end{equation}
Let $\tfrac{1}{2}+iT$ be a zero for $\zeta_\K$, and let $\nu$ be its multiplicity. Then
$f_\K\big(\frac{1}{2}+a+iT\big) \sim 2\nu a^{-1}$ as $a\to 0$. Thus the previous formula overestimates
the multiplicity of zeros by a factor two.
The following argument improves~\eqref{eq:C13} adding greater flexibility to the choice of the weight. Let
$g\colon \R\to \R^+$ be any map and $\mu$ be a measure in $\R$ such that
\begin{equation}\label{eq:C14}
g(\gamma) \leq \int_{\R} \frac{\d\mu(t)}{(\sigma-\tfrac{1}{2})^2+(t-\gamma)^2}
\qquad \forall \gamma\in\R.
\end{equation}
Then summing over all zeros and assuming GRH we have
\begin{align*}
\sum_{\rho} g(\gamma)
&\leq \sum_{\rho}\int_{\R} \frac{\d\mu(t)}{(\sigma-\tfrac{1}{2})^2+(t-\gamma)^2}
 = \frac{1}{2\sigma-1}\int_{\R} \sum_{\rho}\frac{2\sigma-1}{(\sigma-\tfrac{1}{2})^2+(t-\gamma)^2}\d\mu(t),
\end{align*}
producing the bound
\begin{equation}\label{eq:C15}
\sum_{\rho} g(\gamma)
\leq \frac{1}{2\sigma-1}\int_{\R} f_\K(\sigma+it)\d\mu(t).
\end{equation}
Moreover, suppose that $\mu$ is symmetric with respect to a point $t_0$ and such that for some
real $c$ the measure $\mu+c\delta_{t_0}$ is positive, where $\delta_{t_0}$ is the Dirac measure at
$t_0$. Furthermore, let $\tilde{f}_\K(\sigma+it)$ be a concave upper bound for $f_\K(\sigma+it)$ in the
support of $\mu$, then we immediately deduce that
\begin{equation}\label{eq:C16}
\sum_{\rho} g(\gamma)
\leq \frac{\mu(\R)}{2\sigma-1}\,\tilde{f}_\K(\sigma+it_0).
\end{equation}
The argument has a variational flavor: finding the minimum for $\frac{\mu(\R)}{2\sigma-1}$ in the
set of (symmetric around $t_0$ and) positive (outside $t_0$) measures $\mu$ satisfying~\eqref{eq:C14}.

We apply the argument with $g(\gamma) := \chi_{[T-a,T+a]}(\gamma)$ and $T>a$, so that
\[
\sum_{\rho} g(\gamma) = n_\K(T;a).
\]
\smallskip

We have experimented with several possible measures, but actually our best results come from a very
simple choice. In fact, we set
\begin{equation}\label{eq:C17}
\d\mu(t):= \sum_{j=-2}^2c_j\delta_{T-b_j}(t),
\end{equation}
the sum of five Dirac's deltas, with $c_{-j}=c_j$ and $b_{-j}=-b_{j}$ for every $j$, in order to make
$\d\mu$ symmetric around $T$. With this choice~\eqref{eq:C14} becomes
\begin{equation}\label{eq:C18}
\chi_{[-a,a]}(\gamma)
\leq \sum_{j=-2}^2 \frac{c_j}{\alpha^2+(\gamma-b_j)^2}
\qquad \forall \gamma\in\R
\end{equation}
where $\alpha:=\sigma-\frac{1}{2}$; we have also removed the parameter $T$ via the shift
$T+\gamma\to\gamma$. We are interested in a combination of parameters producing a small value for
\[
\frac{\mu(\R)}{2\sigma-1}
= \frac{c_0+2c_1+2c_2}{2\alpha}.
\]
For the moment we have not yet determined any set of values for the parameters, however, we can make a
simple test proving that our strategy has a good chance of producing something interesting.
Suppose that $1/2+iT$ is a zero and let $\nu$ be its multiplicity, then $f_\K(s)\sim 2\Ree
(\nu(s-1/2-iT)^{-1})$ as $s$ goes to $1/2+iT$. By~\eqref{eq:C15} with the measure~\eqref{eq:C17} and
letting $\sigma\to 1/2$ we get
\begin{equation}\label{eq:C19}
n_\K(T;0^+)
 \leq \frac{1}{2\sigma-1}\sum_{j=-2}^2 c_jf_\K(\sigma+iT-ib_j)
 = \nu\sum_{j=-2}^2 \frac{c_j}{\alpha^2+b_j^2} + R(\sigma),
\end{equation}
where the remainder $R(\sigma)$ is $O(\sum |b_j|)$. Hence the function to the right-hand side
of~\eqref{eq:C19} is substantially $\nu$ if we require that the constants $c_j$'s produce an equality
in~\eqref{eq:C18} when $\gamma=0$ and the $b_j$'s are small. In this way we improve on what comes from
the elementary argument~\eqref{eq:C13}.
\smallskip

We have six parameters: $\alpha$, $b_1$ and $b_2$ and the three $c_j's$. Equation~\eqref{eq:C18} shows a
homogeneity in $a$: once we have found a set of parameters $\alpha$, $b_j$ and $c_j$ for $a=1$, the
parameters $a\alpha$, $ab_j$ and $a^2c_j$ can be used for any $a$. We thus suppose $a=1$. Moreover, we
set $\alpha = \tfrac{1}{4}$ and to fix the value of the other parameters we impose the equality
in~\eqref{eq:C18} for $\gamma=0$ (ensuring the good asymptotic estimate for the multiplicity),
$\gamma=\pm \tfrac{3}{5}$ and $\gamma=\pm 1$: values $\frac{1}{4}$ for $\alpha$ and $\pm\frac{3}{5}$ have
been chosen by trial and error and produce our best result, but are essentially arbitrary. The
equalities form a linear system in the $c_j$'s which may be explicitly solved in terms of $b_1$ and
$b_2$.\\
Then we impose two extra conditions: the first one requiring that the zero of the difference of both sides
of~\eqref{eq:C18} in $\gamma=\pm\tfrac{3}{5}$ has at least order two, the second one requiring that the
zero in $\gamma=0$ has at least order four. Due to the symmetry, these conditions correspond to the two
equations
\begin{equation}\label{eq:C20}
\begin{cases}
\displaystyle
     \sum_{j=-2}^2 \frac{c_j(\frac{3}{5}-b_j)}{(\frac{1}{16}+(\frac{3}{5}-b_j)^2)^2} = 0, \\
\displaystyle
     \sum_{j=-2}^2 \frac{c_j(\frac{1}{16}-3b_j^2)}{(\frac{1}{16}+b_j^2)^3} = 0.
\end{cases}
\end{equation}
Observe that the equation
\begin{equation}\label{eq:C21}
\sum_{j=-2}^2 \frac{c_j}{\frac{1}{16}+(\gamma-b_j)^2} = 1
\end{equation}
(which  is an algebraic equation in $\gamma$ of degree $10$) has at least the roots $0$ (multiplicity
$\geq 4$), $\pm \frac{3}{5}$ (multiplicity $\geq 2$, each) and $\pm 1$ (multiplicity $\geq 1$, each).
Hence there are no other roots, and comparing the two sides of~\eqref{eq:C21} as $\gamma\to\infty$ we
conclude that~\eqref{eq:C18} holds in $[-1,1]$. Moreover, \eqref{eq:C18} is also true for $|\gamma|>1$ if
the $c_j$'s are not negative.

Equation~\eqref{eq:C20} can be seen as an algebraic system in the variables $b_1^2$ and $b_2^2$. Using
the resultant of the polynomials, we can check that there is a unique solution such that $0\leq b_1\leq
b_2$ and it is
\[
b_1 = 0.355\ldots,\quad
b_2 = 0.875\ldots,
\]
giving
\[
c_1 = 0.0200\ldots,\quad
c_2 = 0.0491\ldots,\quad
c_3 = 0.0651\ldots.
\]
For this combination we get
\[
\frac{\mu(\R)}{2\sigma-1} \leq \frac{1}{2}
\]
that, coming back to generic $a$, via~\eqref{eq:C16} yields
\begin{equation}\label{eq:C22}
n_\K(T;a)
\leq \frac{a}{2} \tilde{f}_\K\Big(\frac{1}{2} + \frac{a}{4} + iT\Big)
\qquad \forall T>a,
\end{equation}
where for the moment $\tilde{f}_\K$ denotes any concave upper bound of $f_\K$.
\smallskip\\
Our second result~\eqref{eq:C4} is proved with a slightly different argument. It needs the formulas for
the $c_j$'s to be made explicit in terms of the other parameters, thus we further simplify our definition
of the measure imposing $c_2=0$ in~\eqref{eq:C17}, i.e. setting
\[
\d\mu(t):= c_1\delta_{T+b}(t) + c_0\delta_{T}(t) + c_1\delta_{T-b}(t).
\]
We fix the $c_j$'s in such a way as to get an equality in~\eqref{eq:C18} for $\gamma=0$ and
$\gamma=\pm a$; this happens for
\begin{equation}
\label{eq:C23}
\begin{cases}
\displaystyle
c_0 = \frac{-\alpha^6+(3b^2-2a^2)\alpha^4+(3b^2-a^2)a^2\alpha^2}{b^2(5\alpha^2+a^2+b^2)},\\[1em]
\displaystyle
c_1 = \frac{\alpha^6+(2a^2+3b^2)\alpha^4+(a^4+3b^4)\alpha^2+(a^2-b^2)^2b^2}{2b^2(5\alpha^2+a^2+b^2)}.
\end{cases}
\end{equation}
%
%
Formulas in~\eqref{eq:C23} show that $c_1$ is always positive, but that $c_0$ may be negative for some
values of the parameters. However, for these $c_j$'s the function appearing to the right-hand side
of~\eqref{eq:C18} can be written as sum of squares and thus is always positive.
%
%
As a consequence only the range $[-a,a]$ has to be considered for~\eqref{eq:C18}. Previously we have
taken advantage of the homogeneity of the problem in the $a$ parameter, but for the present application
it is useful to work out a formula allowing the limit $a\to 0$ without simultaneously sending $\alpha$ to
$0$. As a consequence we set $b^2 = \frac{a^2}{2}$, giving
\begin{gather*}
\begin{cases}
\displaystyle
c_0 = -2\frac{(2\alpha^2-a^2)(\alpha^2+a^2)\alpha^2}
           {a^2(10\alpha^2+3a^2)}\\
\displaystyle
c_1 = \frac{(2\alpha^2+a^2)(4\alpha^4+12a^2\alpha^2+a^4)}
           {4a^2(10\alpha^2+3a^2)}
\end{cases}
\intertext{which produces}
\frac{c_0+2c_1}{2\alpha}
    = \frac{24\alpha^4+18a^2\alpha^2+a^4}
           {4(10\alpha^2+3a^2)\alpha}
\end{gather*}
%
%
without assuming any proportionality between $\alpha$ and $a$. Once again the choice
$b^2=\frac{a^2}{2}$ is the result of a trial and error procedure. Inequality~\eqref{eq:C18} holds as the
equality
\[
\frac{c_1}{\alpha^2+(\gamma-\frac{a}{\sqrt{2}})^2} + \frac{c_0}{\alpha^2+\gamma^2} + \frac{c_1}{\alpha^2+(\gamma+\frac{a}{\sqrt{2}})^2} = 1
\]
in six complex points (multiplicity included). The definitions of $c_j$'s give the roots $\gamma=0$
(multiplicity $\geq 2$) and $\gamma = \pm a$; the remaining roots solve
\[
\gamma^2 = \frac{a^4-36\alpha^4}{20\alpha^2+6a^2}.
\]
%
%
\noindent%
Inequality~\eqref{eq:C18} is satisfied in the range $\gamma\in[-a,a]$ if and only if these two extra
solutions are either $0$ or non-real.
%
%
This is what happens as long as $a^2< 6\alpha^2$, and since $n_\K(T;0^+) \leq n_\K(T;a)$ for every $a$,
we deduce that
\[
n_\K(T;0^+)
\leq \frac{24\alpha^4+18a^2\alpha^2+a^4}{4(10\alpha^2+3a^2)\alpha} \tilde{f}_\K(\sigma + iT),
\]
where again $\tilde{f}_\K$ denotes any concave upper bound of $f_\K$. Sending $a\to 0$ to the right-hand
side we conclude that
%
\begin{equation}\label{eq:C24}
n_\K(T;0^+)
\leq
\frac{3}{10}(2\sigma-1)\tilde{f}_\K(\sigma + iT)
\qquad \forall T>0,
\qquad \forall \sigma>\frac{1}{2}.
\end{equation}

In the next section we will prove that $f_\K(\sigma+it)\leq \tilde{f}_\K(\sigma+it)$ where $\tilde{f}_\K$
is the function given in Theorem~\ref{th:C1}. With~\eqref{eq:C16} this suffices to prove~\eqref{eq:C3}
and~\eqref{eq:C4} from~\eqref{eq:C22} and~\eqref{eq:C24} respectively, because
$\tilde{f}_\K(\sigma+it)$ is a concave map in the $t$ variable.

\section{Bounds in the critical strip}
\label{sec:C3}
The following two lemmas collect some elementary inequalities involving the gamma function which we will
need later, their proofs are in the Appendix.
\begin{lemma}\label{lem:C1}
Let $s=\sigma+it$ with $\sigma\geq 0$ and $|t|\geq \sigma+2$. Then
\begin{align*}
\Ree\Big(\frac{\Gamma'}{\Gamma}(s)\Big)&\leq \log|s-1/2|.
\end{align*}
\end{lemma}

\begin{lemma}\label{lem:C2}
Let $u\in[-3/4,-1/4]$, then
\begin{subequations}\label{eq:C25}
\begin{align}
\int_\R |\Gamma(u+iy)|\,d y
&\leq 4.73.                                                   \label{eq:C25a}\\
\intertext{Furthermore suppose $|t|\geq 10$, then}
\int_\R |\Gamma(u+i(t-y))|\log(1+|y|) \d y
&\leq 4.73\log(1+|t|),                                        \label{eq:C25b}\\
\int_\R \frac{|\Gamma(u+i(t-y))|}{|1+4iy|}\,d y
&\leq 0.171,                                                  \label{eq:C25c}\\
\int_\R \frac{|\Gamma(u+i(t-y))|}{1+\alpha y^2}\,d y
&\leq
\begin{cases}
0.013         &\text{if }\alpha = 1,\\
0.007         &\text{if }\alpha = 2.
\end{cases}                                                   \label{eq:C25d}
\end{align}
\end{subequations}
\end{lemma}
%

Let $R_\K(T)$ be defined via
\begin{equation}\label{eq:C26}
N_\K(T) =: \frac{T}{\pi}\log\Big(\Big(\frac{T}{2\pi e}\Big)^{n_\K}\disc_\K\Big) + R_\K(T).
\end{equation}
Trudgian~\cite[Th.~2]{TrudgianIII} proved unconditionally that for every $T\geq 1$ one has
\begin{equation}\label{eq:C27}
|R_\K(T)| \leq \tilde{R}_\K(T) := d_1 W_\K(T) + d_2n_\K + d_3
\end{equation}
with
\[
W_\K(T) := \log\disc_\K + n_\K \log\Big(\frac{T}{2\pi}\Big)
\]
and $d_1= 0.317$, $d_2= 6.333+0.317\log(2\pi) \leq 6.9157$ and $d_3 = 3.482$. We use this result first to
bound $n_\K(t;1)$, and later to bound certain finite sums over zeros (see Lemma~\ref{lem:C4} below).
\begin{lemma}\label{lem:C3}
For $t\in\R$,
\begin{subequations}\label{eq:C28}
\begin{align}
n_\K(t;1) &\leq  0.636 W_\K(t) + 6.92n_\K + 3.49
&& \forall|t|> 1,                                   \label{eq:C28a}\\
n_\K(t;1) &\leq 0.954\log\disc_\K + 5.19n_\K + 3.49
&& \forall|t|\leq 1.                                \label{eq:C28b}
\end{align}
\end{subequations}
\end{lemma}
\begin{proof}
The symmetry of roots allows us to assume $t\geq 0$.
For $t\geq 2$, using~\eqref{eq:C27}:
\begin{align*}
n_\K(t;1) &= \frac{1}{2}(N_\K(t+1^+)-N_\K(t-1^-))\\
          &\leq \frac{1}{2\pi}\Big((t+1)\log\Big(\Big(\frac{t+1}{2\pi e}\Big)^{n_\K}\disc_\K\Big)
                                  -(t-1)\log\Big(\Big(\frac{t-1}{2\pi e}\Big)^{n_\K}\disc_\K\Big)
                              \Big)\\
          &   + \frac{d_1}{2}(W_\K(t+1)+W_\K(t-1)) + d_2n_\K + d_3.
\intertext{Since $(t+1)\log(t+1)-(t-1)\log(t-1)\leq 2(1+\log t)$ and $W_\K(t)$ is a concave map, we get}
          &\leq \frac{1}{2\pi}(2n_\K(1 + \log t)
                              + 2\log\disc_\K
                              - 2n_\K \log(2\pi e)
                          )
              + d_1 W_\K(t) + d_2n_\K + d_3 \\
          &= \Big(\frac{1}{\pi}+d_1\Big) W_\K(t) + d_2n_\K + d_3.
\end{align*}
%
Introducing the values for $d_j$'s we get~\eqref{eq:C28a} for $t\geq 2$.\\
For $1< t\leq 2$, we have
\begin{align*}
n_\K(t;1) &\leq \frac{1}{2}N_\K(t+1^+)
           \leq \frac{1}{2}N_\K(3)\\
          &\leq \Big(\frac{3}{2\pi}+\frac{d_1}{2}\Big)\log\disc_\K
              +\Big(\frac{3}{2\pi}\log\Big(\frac{3}{2\pi e}\Big) + \frac{d_1}{2}\log\Big(\frac{3}{2\pi}\Big) + \frac{d_2}{2}\Big)n_\K
              + \frac{d_3}{2}.
\end{align*}
%
This bound is a bit smaller than what we get extrapolating to $t\in(1,2]$ the formula we have found for $t>2$.
Not needing its full strength, we estimate $n_\K(t;1)$ for $t\in(1,2]$ with that bound, thus proving~\eqref{eq:C28a}.\\
For $0\leq t\leq 1$, we have
\begin{align*}
n_\K(t;1) &\leq N_\K(t+1^+)
           \leq N_\K(2)\\
%
%
          &\leq \Big(\frac{2}{\pi}+d_1\Big)\log\disc_\K
              +\Big(\frac{2}{\pi}\log\Big(\frac{1}{\pi e}\Big) + d_1\log\Big(\frac{1}{\pi}\Big) + d_2\Big)n_\K
              + d_3
\end{align*}
giving~\eqref{eq:C28b} when the values for $d_j$'s are introduced.
\end{proof}

\begin{lemma}\label{lem:C4}
Let $c>0$ and $t\in\R$, with $|t|>c+1$. Let $u>0$, then
\begin{align*}
\sum_{|\gamma-t|\leq c} \frac{1}{|u + i(\gamma-t)|}
\leq \Big(\frac{\asinh(c/u)}{\pi} + \frac{d_1}{u}\Big)W_\K(|t|) + \frac{d_2 n_\K + d_3}{u}.
\end{align*}
\end{lemma}

\begin{proof}
Without loss of generality we can assume $t>0$. We write the sum as an integral in the density of zeros:
\begin{align*}
\sum_{|\gamma-t|\leq c} \frac{1}{|u + i(\gamma-t)|}
&= \int_{t-c^-}^{t+c^+} \frac{\d(N_\K(\gamma))}{2|u + i(\gamma-t)|},
\end{align*}
where the factor $\tfrac{1}{2}$ appears because only zeros with positive imaginary part matter, since
$t-c>1$. By~\eqref{eq:C26} this is
\begin{align*}
&= \int_{-c}^{c} \frac{W_\K(\gamma+t)}{2\pi|u + i\gamma|} \d\gamma
   +\int_{-c^-}^{c^+}\frac{\d(R_\K(\gamma+t))}{2|u + i\gamma|}.
\end{align*}
$W_\K$ is a concave map, thus
\begin{align*}
\int_{-c}^{c} \frac{W_\K(\gamma+t)}{2\pi|u + i\gamma|} \d\gamma
\leq \int_{-c}^{c} \frac{\d\gamma}{2\pi|u + i\gamma|} W_\K(t)
 =   \frac{\asinh(c/u)}{\pi} W_\K(t).
\end{align*}
Moreover, integrating by parts and using the upper bound $|R_\K| \leq \tilde{R}_\K$ in~\eqref{eq:C27} we
get
\begin{align*}
\Big|\int_{-c^-}^{c^+}\frac{\d(R_\K(\gamma+t))}{2|u + i\gamma|}\Big|
&\leq \frac{\tilde{R}_\K(t+c) + \tilde{R}_\K(t-c)}{2(u + c^2)^{1/2}}
   + \frac{1}{2}\int_{-c}^{c}\tilde{R}_\K(\gamma+t) \frac{|\gamma|\d \gamma}{(u^2 + \gamma^2)^{3/2}}
\intertext{and since $\tilde{R}_\K$ is also a concave map for positive arguments, we get}
&\leq \frac{\tilde{R}_\K(t)}{(u + c^2)^{1/2}}
   + \frac{\tilde{R}_\K(t)}{2}\int_{-c}^{c}\frac{|\gamma|\d \gamma}{(u^2 + \gamma^2)^{3/2}}
%
%
  =\frac{\tilde{R}_\K(t)}{u}.
\end{align*}
\end{proof}

\begin{lemma}\label{lem:C5}
For every real $t$ one has:
\[
\Big|\frac{\Gamma'_\K}{\Gamma_\K}\Big(\frac{1}{4}+it\Big)-\frac{\Gamma'_\K}{\Gamma_\K}(2+it)\Big|
\leq \frac{10 n_\K}{|1+4it|}.
\]
\end{lemma}
\begin{proof}
From the definition of $\Gamma_\K$ in~\eqref{eq:C9} we have
\begin{multline*}
\frac{\Gamma'_\K}{\Gamma_\K}\Big(\frac{1}{4}+it\Big)-\frac{\Gamma'_\K}{\Gamma_\K}(2+it)
  = r_2       \Big(\frac{\Gamma'}{\Gamma}\Big(\frac{5}{8} +i\frac{t}{2}\Big)
                   -\frac{\Gamma'}{\Gamma}\Big(\frac{3}{2}+i\frac{t}{2}\Big)
              \Big)                                             \\
   +(r_1+r_2) \Big(\frac{\Gamma'}{\Gamma}\Big(\frac{1}{8}+i\frac{t}{2}\Big)
                   -\frac{\Gamma'}{\Gamma}\Big(1+i\frac{t}{2}\Big)
              \Big).
\end{multline*}
From the functional equation $s\Gamma(s)=\Gamma(s+1)$ we get
\begin{equation}\label{eq:C29}
\frac{\Gamma'}{\Gamma}\Big(\frac{1}{8}+i\frac{t}{2}\Big)
                    -\frac{\Gamma'}{\Gamma}\Big(1+i\frac{t}{2}\Big)
= - \frac{8}{1+4it} + \frac{\Gamma'}{\Gamma}\Big(\frac{9}{8}+i\frac{t}{2}\Big)
                    -\frac{\Gamma'}{\Gamma}\Big(1+i\frac{t}{2}\Big),
\end{equation}
now we proceed as in Lemma~\ref{lem:C1}: from the Euler--Maclaurin summation formula for $\log\Gamma(s)$
the difference of the logarithmic derivatives is at most
\[
\Big|\log\Big(\frac{\frac{9}{8}+i\frac{t}{2}}{1+i\frac{t}{2}}\Big)\Big|
+ \frac{1}{2} \frac{1/8}{|\frac{9}{8}+i\frac{t}{2}|\cdot|1+i\frac{t}{2}|}
+ \frac{1}{12}\Big|\frac{1}{(\frac{9}{8}+i\frac{t}{2})^2}-\frac{1}{(1+i\frac{t}{2})^2}\Big|
+ \frac{1}{3}\int_0^{+\infty}\frac{\d u}{|u+1+it/2|^3}.
\]
The absolute value of the logarithm is lower than $-\log\big(1-\frac{1/8}{|1+i\frac{t}{2}|}\big)$, and the
integral is $4/(4+t^2+2\sqrt{4+t^2})$. Using these bounds one proves that~\eqref{eq:C29} is bounded by
$10/|1+4it|$.
%
%
%
%
The same argument applied to the other difference of gamma functions completes the proof.
\end{proof}

In order to prove that $f_\K(\sigma+it)\leq \tilde{f}_\K(\sigma+it)$ we need an upper bound for
$-\frac{\zeta'_\K}{\zeta_\K}(s)$ in the critical strip, and for this purpose we follow the argument used
in~\cite[Th. 14.4]{Titchmarsh1} for the Riemann zeta function. In our setting, however, the argument will
be considerably complicated by the need of good explicit constants. Let $\sigma =
\Ree(s)\in(\tfrac{1}{2},1)$ and let $\delta$ be a parameter in $(0,1)$. We get
\[
-\sum_{n=1}^\infty \frac{\tilde{\Lambda}_\K(n)}{n^s}e^{-\delta n}
=
\frac{1}{2\pi i}\int_{2-i\infty}^{2+i\infty} \frac{\zeta'_\K}{\zeta_\K}(z) \Gamma(z-s)\delta^{s-z}\d z.
\]
Moving the integration line to $\Ree(z)=\frac{1}{4}$ we get the equality
\begin{align}
-\frac{\zeta'_\K}{\zeta_\K}(s)
 ={}& \sum_{n=1}^\infty \frac{\tilde{\Lambda}_\K(n)}{n^s}e^{-\delta n}
 -    \Gamma(1-s)\delta^{s-1}
 +    \sum_\rho \Gamma(\rho-s)\delta^{s-\rho}
 +    \frac{1}{2\pi i}\int_{1/4-i\infty}^{1/4+i\infty} \frac{\zeta'_\K}{\zeta_\K}(z) \Gamma(z-s)\delta^{s-z}\d z   \notag\\
:={}& I + II + III + IV;                                                                                          \label{eq:C30}
\end{align}
here $I$ is the value of the original integral, $II$ comes from the pole of $\zeta_\K$ at $z=1$, $III$
from the nontrivial zeros, and $-\frac{\zeta'_\K}{\zeta_\K}(s)$ from the pole of $\Gamma(z-s)$ at $z=s$:
the Cauchy theorem is applicable here since $-\frac{\zeta'_\K}{\zeta_\K}(s)$ grows polynomially along the
vertical lines, while the gamma function decays exponentially. The following lemmas provide bounds for
$I$--$IV$ and will be combined into a suitable bound for $\frac{\zeta'_\K}{\zeta_\K}(s)$ in
Lemma~\ref{lem:C12}.

\begin{lemma}[Bound of $I$]\label{lem:C6}
Assume RH. Let $\sigma\in(\frac{1}{2},1)$ and $\delta>0$. Then
\[
\sum_{n=1}^\infty \frac{\tilde{\Lambda}_\K(n)}{n^\sigma}\,e^{-\delta n}
\leq n_\K\sum_{n=1}^\infty \frac{\Lambda(n)}{n^\sigma}\,e^{-\delta n}
\leq \Big(\frac{\delta^{\sigma-1}}{1-\sigma}
          + \frac{0.07}{2\sigma-1}
          + 4
     \Big)n_\K.
\]
\end{lemma}
\begin{proof}
The first inequality is an immediate consequence of the inequality $\tilde{\Lambda}_\K(n) \leq
n_\K\Lambda(n)$. For the second one, let $\psi^{(1)}(x) := \int_0^x\psi(u)\d u = \sum_{n\leq
x}\Lambda(n)(x-n)$. Then
\begin{equation}\label{eq:C31}
\Big|\psi^{(1)}(x) - \frac{x^2}{2}\Big|
\leq   0.0462 x^{3/2}
     + 1.838  x
\qquad \forall x\geq 1,
\end{equation}
(see~\cite[Th.~1.1]{GrenieMolteni2}; there the claim is stated for $x\geq 3$, but it actually holds for
$x\geq 1$).
Thus, by partial summation we get
\begin{align*}
\sum_{n=1}^\infty \frac{\Lambda(n)}{n^\sigma}\,e^{-\delta n}
&= \int_{2^-}^{+\infty} \psi^{(1)}(x)\Big[\frac{e^{-\delta x}}{x^\sigma}\Big]'' \d x
\intertext{which we write as}
&= \int_{2}^{+\infty} \frac{x^2}{2}\Big[\frac{e^{-\delta x}}{x^\sigma}\Big]'' \d x
 + \int_{2^-}^{+\infty} \Big[\psi^{(1)}(x)-\frac{x^2}{2}\Big]\Big[\frac{e^{-\delta x}}{x^\sigma}\Big]'' \d x\\
%
&= (2+\sigma+2\delta)\frac{e^{-2\delta}}{2^\sigma}
   + \delta^{\sigma-1}\int_{2\delta}^{+\infty} \frac{e^{-u}}{u^\sigma} \d u
 + \int_{2^-}^{+\infty} \Big[\psi^{(1)}(x)-\frac{x^2}{2}\Big]\Big[\frac{e^{-\delta x}}{x^\sigma}\Big]'' \d x\\
%
%
&\leq (2+\sigma+2\delta)\frac{e^{-2\delta}}{2^\sigma}
   + \delta^{\sigma-1}\Gamma(1-\sigma)
   + \int_{2^-}^{+\infty} \Big[\psi^{(1)}(x)-\frac{x^2}{2}\Big]\Big[\frac{e^{-\delta x}}{x^\sigma}\Big]'' \d x.
\end{align*}
\enlargethispage{3\baselineskip}
The function $e^{-\delta x}x^{-\sigma}$ is completely monotone, thus
\begin{align*}
\int_2^{+\infty} x^{3/2}\Big|\Big[\frac{e^{-\delta x}}{x^\sigma}\Big]''\Big| \d x
&= \int_2^{+\infty} x^{3/2}\Big[\frac{e^{-\delta x}}{x^\sigma}\Big]'' \d x \\
%
&= \sqrt{2}\Big(\frac{3}{2}+\sigma+2\delta\Big)\frac{e^{-2\delta}}{2^\sigma}
  + \frac{3}{4}\int_2^{+\infty} e^{-\delta x} x^{-1/2-\sigma}\d x.
\intertext{The last integral is at most
$\min(e^{-2\delta}\delta^{-1},e^{-2\delta}(\sigma-1/2)^{-1})$, but later we will choose $\sigma$ and
$\delta$ such that the minimum comes from the term in $\sigma$, thus we write}
&\leq \sqrt{2}\Big(\frac{3}{2}+\sigma+2\delta\Big)\frac{e^{-2\delta}}{2^\sigma}
  + \frac{3}{2}\frac{e^{-2\delta}}{2\sigma-1}.
\end{align*}
Moreover,
\[
\int_2^{+\infty} x\Big|\Big[\frac{e^{-\delta x}}{x^\sigma}\Big]''\Big| \d x
 = \int_2^{+\infty} x\Big[\frac{e^{-\delta x}}{x^\sigma}\Big]'' \d x
 = (1+\sigma+2\delta)\frac{e^{-2\delta}}{2^\sigma},
\]
hence, recalling~\eqref{eq:C31} and using the inequality $\Gamma(1-\sigma)\leq (1-\sigma)^{-1}$ for
$\sigma\in(0,1)$, we get
\begin{align*}
\sum_{n=1}^\infty \frac{\Lambda(n)}{n^\sigma}\,e^{-\delta n}
\leq& (2+\sigma+2\delta)\frac{e^{-2\delta}}{2^\sigma}
     + \frac{\delta^{\sigma-1}}{1-\sigma}
     + 0.0462 \Big(\sqrt{2}\Big(\frac{3}{2}+\sigma+2\delta\Big)\frac{e^{-2\delta}}{2^\sigma}
                   + \frac{3}{2}\frac{e^{-2\delta}}{2\sigma-1}\Big)             \\
    &+ 1.838  (1+\sigma+2\delta)\frac{e^{-2\delta}}{2^\sigma}
\leq  \frac{\delta^{\sigma-1}}{1-\sigma}
     + \frac{0.07}{2\sigma-1}
     + 4.
\end{align*}
%
%
\end{proof}

\begin{lemma}[Bound of $II$]\label{lem:C7}
Let $\sigma \in(\frac{1}{2},1)$, $|t|\geq 2$ and $\delta\in(0,1)$.
Then
\[
|\Gamma(1-s)\delta^{s-1}| \leq \sqrt{2\pi}
e^{-\frac{\pi}{2}|t|}\delta^{\sigma-1}.
\]
\end{lemma}
\begin{proof}
It is sufficient to prove that $|\Gamma(s)e^{\pi s/2}| \leq \sqrt{2\pi}$ for any $s\in D:=\{s\colon
\Ree(s)\in[0,\tfrac{1}{2}],\ \Imm(s)\geq 10\}$. By the Phragm\'{e}n-Lindel\"{o}f principle it is sufficient to
prove it for $s\in\partial D$. The claim for $\Ree(s)=\tfrac{1}{2}$ follows immediately from the
reflection formula $\Gamma(s)\Gamma(1-s) = \frac{\pi}{\sin\pi s}$. The claim for the other two lines may
be proved using the Euler--Maclaurin formula for the gamma function.
%
%
\end{proof}

\begin{lemma}[Bound of $III$]\label{lem:C8}
Assume GRH. Let $\sigma\in(\frac{1}{2},1)$, $|t|\geq 10$ and $\delta\in(0,1)$. Then
\[
\Big|\sum_\rho \Gamma(s-\rho)\delta^{s-\rho}\Big|
 \!\leq\! \delta^{\sigma-\tfrac{1}{2}}
      \Big[\Big(\frac{\log(\frac{1}{2\sigma-1})}{\pi} + \frac{0.64}{2\sigma-1} + 0.82\Big) W_\K(t)
           + \Big(\frac{13.9}{2\sigma-1} + 1.6\Big) n_\K
           + \frac{6.9}{2\sigma-1}
           + 0.8
      \Big].
\]
\end{lemma}
\begin{proof}
We are assuming GRH, thus $\big|\sum_\rho \Gamma(s-\rho)\delta^{s-\rho}\big|
\leq\delta^{\sigma-\tfrac{1}{2}} \sum_\rho |\Gamma(s-\rho)|$. We separate the contribution of zeros
close to $t$, since in this case the weight $\Gamma(s-\rho)$ is large because of the pole of $\Gamma$
at $0$. We chose $2$ as threshold value, which appears to be near the optimal value $\approx 2.3$.
Since $|(u+iv)\Gamma(u+iv)|\leq 1$ for $u\in(0,1/2)$ and every $v\in\R$, we have for every $t>3$ (and
setting $u:=\sigma -\tfrac{1}{2} \in (0,\tfrac{1}{2})$)
\begin{align*}
\sum_{|\gamma-t|\leq 2} |\Gamma(u + i(\gamma-t))|
\leq \sum_{|\gamma-t|\leq 2} \frac{1}{|u + i(\gamma-t)|}
\end{align*}
so that by Lemma~\ref{lem:C4} we get
\begin{equation}\label{eq:C32}
\sum_{|\gamma-t|\leq 2} |\Gamma(u + i(\gamma-t))|
 \leq \Big(\frac{\asinh(\frac{4}{2\sigma-1})}{\pi} + \frac{2d_1}{2\sigma-1}\Big) W_\K(t) + \frac{2d_2n_\K + 2d_3}{2\sigma-1}.
\end{equation}
To estimate the contribution of zeros with $|\gamma-t|\geq 2$ we use the bound $|\Gamma(u+iv)|\leq
\sqrt{2\pi} e^{-\frac{\pi}{2}|v|}$ for $u\in(0,1/2)$, $|v|\geq 2$ proved in Lemma~\ref{lem:C7}.
%
%
Thus we get (assuming $t>3$) that
\[
\sum_{|\gamma-t|\geq 2} |\Gamma(s-\rho)|
\leq \sqrt{2\pi}e^{-\pi}\sum_{j=0}^\infty e^{-j\pi}\big(n_\K(t+2j+3;1)+n_\K(t-2j-3;1)\big).
\]
Without loss of generality we can assume $t\in\R\backslash\Z$; then the claim for $t\in\Z$ will follow by
continuity. Under this hypothesis the quantity $|t-2j-3|$ is smaller than $1$ only for
$j=\jbar:=\intpart{\frac{t-2}{2}}$. Thus from~\eqref{eq:C28a} and~\eqref{eq:C28b} we deduce that
\begin{align}
\sum_{|\gamma-t|\geq 2} |\Gamma(s-\rho)|
\leq&   \sqrt{2\pi}e^{-\pi}\sum_{j=0}^\infty e^{-j\pi}(0.64\log\disc_\K + (0.64\log(t+2j+3) + 5.75)n_\K + 3.49)\notag\\
    &+ \sqrt{2\pi}e^{-\pi}\sum_{\substack{j=0\\ j\neq \bar\jmath}}^\infty
                                           e^{-j\pi}(0.64\log\disc_\K + (0.64\log(|t-2j-3|) + 5.75)n_\K + 3.49)\notag\\
    &+ \sqrt{2\pi}e^{-\pi}e^{-\jbar\pi}(0.96\log\disc_\K + 5.19n_\K + 3.49)                                    \notag\\
\leq&  \frac{\sqrt{2\pi}e^{-\pi}}{1-e^{-\pi}}(1.28\log\disc_\K + 11.5n_\K + 6.98)
     + 0.32\sqrt{2\pi}e^{-\frac{\pi}{2}(t-2)}\log\disc_\K                                                      \notag\\
    &+ 0.64\sqrt{2\pi}e^{-\pi} \Big(\sum_{j=0}^\infty e^{-j\pi}\log(t+2j+3)
                                    + \sum_{\substack{j=0\\ j\neq\jbar}}^\infty e^{-j\pi}\log(|t-2j-3|)
                               \Big)n_\K.                                                                      \label{eq:C33}
\end{align}
To bound the sums we use the inequalities $\log(t+2j+3) \leq \log t + \frac{2j+3}{t}$ for the first one,
and $\log(|t-2j-3|) \leq \log t$ when $j\leq J$ and $\log(|t-2j-3|) \leq \log t + \frac{2(j-J)}{t}$
when $j>J$ for the second, with $J:= \intpart{t-\frac{3}{2}}$. Thus,
\begin{align*}
\sum_{j=0}^\infty e^{-j\pi}\log(t+2j+3)
    &+ \sum_{\substack{j=0\\ j\neq\bar{\jmath}}}^\infty e^{-j\pi}\log(|t-2j-3|)    \\
\leq&  \sum_{j=0}^\infty e^{-j\pi}\Big(\log t + \frac{2j+3}{t}\Big)
                                   + \sum_{j=0}^{\infty} e^{-j\pi}\log t
                                   + \frac{2}{t}\sum_{j=J+1}^\infty e^{-j\pi}(j-J) \\
  = &  \frac{2\log t +3}{1-e^{-\pi}}
     + 2\frac{e^{-\pi}(1+e^{-J\pi})}{t(1-e^{-\pi})^2}.
\end{align*}
Moving this bound into~\eqref{eq:C33} we get for $t\geq 10$ that
\begin{align*}
\sum_{|\gamma-t|\geq 2} |\Gamma(s-\rho)|
\leq& 0.145\log\disc_\K + (0.145\log t + 1.33)n_\K + 0.8
%
%
\leq 0.15W_\K(t) + 1.6n_\K + 0.8,
%
%
\end{align*}
that with~\eqref{eq:C32} gives
\begin{align*}
\sum_{\rho} |\Gamma(s -\rho)|
 \leq& \Big(\frac{\asinh(\frac{4}{2\sigma-1})}{\pi} + \frac{2d_1}{2\sigma-1} + 0.15\Big) W_\K(t) + \frac{2d_2n_\K + 2d_3}{2\sigma-1}
      + 1.6n_\K + 0.8.
\end{align*}
We get the claim using the bound $\asinh(4/u) \leq \log(1/u) + \asinh 4$ which holds for $u\in(0,1]$, and
the known values for $d_j$'s.
%
\end{proof}

To bound $IV$ efficiently we split it in two
\begin{align*}
IV &= \frac{1}{2\pi i}\int_{1/4-i\infty}^{1/4+i\infty}\!\!
                       \sum_{\substack{|\gamma-y|\leq 1}}\!\!\frac{\Gamma(z-s)}{z-\rho} \delta^{s-z}\d z
     +\frac{1}{2\pi i}\int_{1/4-i\infty}^{1/4+i\infty}
                       \Big[\frac{\zeta'_\K}{\zeta_\K}(z) - \sum_{\substack{|\gamma-y|\leq 1}}\frac{1}{z-\rho}\Big] \Gamma(z-s)\delta^{s-z}\d z\\
   &=: IVa + IVb,
\end{align*}
where $y:=\Imm(z)$, which we estimate separately.

\begin{lemma}[Bound of $IVa$]\label{lem:C9}
Assume GRH. Let $\sigma \in(\frac{1}{2},1)$, $|t|\geq 10$ and $\delta\in(0,1)$. Then
\[
|IVa| \leq \frac{\delta^{\sigma-1/4}}{2\pi}\big(9.16\log\disc_\K + (9.16\log(|t|+1) + 114.03)n_\K + 65.88\big).
\]
\end{lemma}
\begin{proof}
It is sufficient to prove that
\[
\int_{\R} \sum_{\substack{|\gamma-y|\leq1}}\frac{|\Gamma(u+i(y-t))|}{|\frac{1}{4}+i(\gamma-y)|}\d y
\leq 9.16\log\disc_\K + (9.16\log(|t|+1) + 89.73)n_\K + 65.88
\]
when $u\in[-\frac{3}{4},-\frac{1}{4}]$ and $t\geq 10$.\\
We split the integral into the regions $|y|\geq 2$ where the sum on the zeros is estimated by
Lemma~\ref{lem:C4} with $c=1$, and the remaining part $|y|< 2$ where the sum is estimated simply by
$4n_\K(y;1)$, getting
\begin{align}
\int_{\R} &\sum_{\substack{|\gamma-y|\leq1}}\frac{|\Gamma(u+i(y-t))|}{|\frac{1}{4}+i(\gamma-y)|}\d y
\leq \int_{|y|<2}     4n_\K(y;1) |\Gamma(u+i(y-t))|\d y                                                                           \notag \\
&    +\int_{|y|\geq 2} \Big(\Big(\frac{\asinh 4}{\pi} + 4d_1\Big)W_\K(|y|) + 4(d_2 n_\K + d_3)\Big) |\Gamma(u+i(y-t))|\d y.       \notag
\intertext{Now we restore the part of the integral with $|y|<2$, getting}
&=\int_{|y|<2} \Big(4n_\K(y;1) - \Big(\frac{\asinh 4}{\pi} + 4d_1\Big)W_\K(|y|+1) - 4(d_2 n_\K + d_3)\Big) |\Gamma(u+i(y-t))|\d y \label{eq:C34}\\
&    +\int_{\R}    \Big(\Big(\frac{\asinh 4}{\pi} + 4d_1\Big)W_\K(|y|+1) + 4(d_2 n_\K + d_3)\Big) |\Gamma(u+i(y-t))|\d y.         \notag
\end{align}
The exponential decay of the gamma function and the assumption $t\geq 10$ allow us to bound the first
integral trivially, without affecting the strength of the result. From~\eqref{eq:C28a} and
\eqref{eq:C28b} we get
\[
4n_\K(y;1) - \Big(\frac{\asinh 4}{\pi} + 4d_1\Big)W_\K(|y|+1) - 4(d_2 n_\K + d_3)
\leq 2\log\disc_\K
\]
in $y<2$.
%
%
Moreover, $|\Gamma(u+iv)|\leq 2\cdot 10^{-6}$ when $u\in[-\frac{3}{4},-\frac{1}{4}]$ and $|v|\geq 8$.
With bounds~\eqref{eq:C25a} and~\eqref{eq:C25b} these facts prove that~\eqref{eq:C34} is bounded by
\begin{align*}
2\cdot 10^{-5}\log\disc_\K
    +4.73\Big(\Big(\frac{\asinh 4}{\pi} + 4d_1\Big)W_\K(|t|+1) + 4(d_2 n_\K + d_3)\Big)
\end{align*}
%
which is the claim, once the values for $d_j$'s are considered.
%
%
\end{proof}

The following lemma bounds the integrand in $IVb$.

\begin{lemma}\label{lem:C10}
Assume GRH. For $s= \frac{1}{4} + it$ with $t\not\in\Z$ we have
\begin{multline*}
\Big|\frac{\zeta'_\K}{\zeta_\K}(s)-\sum_{\substack{\rho\\|\gamma-t|\leq1}}\frac{1}{s-\rho}\Big|
\leq
 \Big(2.18 + \frac{3.2}{1+t^2}\Big)\log\disc_\K
     + 11.7
     + \frac{7}{1+2t^2} \\
     + \Big(2.18\log(|t|+1) + 21.6 + \frac{10}{|1+4it|}\Big)n_\K.
\end{multline*}
\end{lemma}

\begin{proof}
We can assume $t>0$. We subtract~\eqref{eq:C11} at $s=\frac{1}{4} + it$ and $2 + it$, obtaining
\[
\frac{\zeta'_\K}{\zeta_\K}(s)-\frac{\zeta'_\K}{\zeta_\K}(2+it)
 = \sum_\rho\Big(\frac{1}{s-\rho}-\frac{1}{2+it-\rho}\Big)
   -\frac{\Gamma'_\K}{\Gamma_\K}(s) + \frac{\Gamma'_\K}{\Gamma_\K}(2+it)
   -\Big(\frac{1}{s}+\frac{1}{s-1}-\frac{1}{2+it}-\frac{1}{1+it}\Big).
\]
We use~\eqref{eq:C6} and Lemma~\ref{lem:C5} to estimate $\frac{\zeta'_\K}{\zeta_\K}(2+it)$ and the
gamma factors respectively, and the bound $|\frac{1}{s}-\frac{1}{2+it} + \frac{1}{s-1} - \frac{1}{1+it}|
\leq \frac{7}{1+2t^2}$.
%
In this way we get
\begin{multline}\label{eq:C35}
\Big|\frac{\zeta'_\K}{\zeta_\K}(s)-\sum_{\substack{|\gamma-t|\leq1}}\frac{1}{s-\rho}\Big|
\leq
 n_\K \Big|\frac{\zeta'}{\zeta}(2)\Big|
 + \sum_{\substack{|\gamma-t|>1}}\Big|\frac{1}{s-\rho}-\frac{1}{2+it-\rho}\Big|     \\
 + \sum_{\substack{|\gamma-t|\leq1}}\Big|\frac{1}{2+it-\rho}\Big|
 + \frac{10 n_\K}{|1+4it|}
 + \frac{7}{1+2t^2}.
\end{multline}
%
Moreover, for the first sum on the right-hand side of~\eqref{eq:C35} we have
\begin{align}
\frac{4}{7}\sum_{\substack{|\gamma-t|>1}}\Big|\frac{1}{s-\rho}-\frac{1}{2+it-\rho}\Big|
    &  =   \sum_{\substack{|\gamma-t|>1}}\frac{1}{|s-\rho||2+it-\rho|}                  \notag\\
    &\leq  \sum_{j=1}^\infty \frac{n_\K(t+2j;1)+n_\K(t-2j;1)}{|\frac{1}{4}+i(2j-1)||\tfrac{3}{2}+i(2j-1)|}.
\label{eq:C36}
\end{align}
By hypothesis $t$ is not an integer, then $|t-2j|$ is in $(0,1)$ only for $j = \jbar :=
\intpart{(t+1)/2}$. Thus using the bound in~\eqref{eq:C28a} for $j\neq \jbar$ and~\eqref{eq:C28b} when
$j=\jbar$ we deduce that~\eqref{eq:C36} is
\begin{align*}
\leq& \sum_{\substack{j=1\\ j\neq \jbar}}^\infty \frac{0.64 W_\K(t+2j) + 0.64 W_\K(|t-2j|) + 2\cdot6.92n_\K + 2\cdot 3.49}{|\frac{1}{4}+i(2j-1)||\tfrac{3}{2}+i(2j-1)|}\\
    &+\frac{0.64 W_\K(t+2j) + 6.92n_\K + 3.49 + (0.96\log\disc_\K + 5.19n_\K + 3.49)}{|\frac{1}{4}+i(2j-1)||\tfrac{3}{2}+i(2j-1)|}\Big|_{j=\jbar \text{ if $t>1$}}.
\end{align*}
Using $\log(t+2j)+\log(|t-2j|) \leq 2\log(2(t+1)j)$ (for $t\geq 0$ and $j\geq 1$):
\begin{align*}
\leq& 2\sum_{\substack{j=1\\ j\neq \jbar}}^\infty \frac{0.64 (\log\disc_\K + \log(2(t+1)j/2\pi)n_\K) + 6.92n_\K + 3.49}{|\frac{1}{4}+i(2j-1)||\tfrac{3}{2}+i(2j-1)|}\\
    &+\frac{0.64 (\log\disc_\K + \log((t+2j)/2\pi)n_\K) + 6.92n_\K + 3.49 + (0.96\log\disc_\K + 5.19n_\K + 3.49)}{|\frac{1}{4}+i(2j-1)||\tfrac{3}{2}+i(2j-1)|}\Big|_{j=\jbar \text{ if $t>1$}}.
\end{align*}
\noindent
Suppose $t>1$. Then restoring the missing term in the first sum we get that~\eqref{eq:C36} is bounded
above by
\begin{align*}
\leq& 2\sum_{j=1}^\infty \frac{0.64 (\log\disc_\K + \log(2(t+1)j/2\pi)n_\K) + 6.92n_\K + 3.49}{|\frac{1}{4}+i(2j-1)||\tfrac{3}{2}+i(2j-1)|}\\
    &+\frac{0.64(\log(2\pi(t+2j))-2\log(2(t+1)j))n_\K + (0.96-0.64)\log\disc_\K + (5.19-6.92)n_\K}{|\frac{1}{4}+i(2j-1)||\tfrac{3}{2}+i(2j-1)|}\Big|_{j=\jbar}.
\end{align*}
\noindent
Since
\[
\sum_{j=1}^\infty \frac{1}{|\frac{1}{4}+i(2j-1)||\tfrac{3}{2}+i(2j-1)|} \leq 0.76,
\qquad
\sum_{j=1}^\infty \frac{\log(2j)}{|\frac{1}{4}+i(2j-1)||\tfrac{3}{2}+i(2j-1)|} \leq 0.82
\]
%
and since for $t\geq 1$ one has
\[
\log(2\pi(t+2j))-2\log(2(t+1)j)+5.19-6.92 < 0
\qquad \text{for}\quad j=\jbar,
\]
\[
\frac{1}{|\frac{1}{4}+i(2j-1)||\tfrac{3}{2}+i(2j-1)|} \leq \frac{5.4}{1+t^2}
\qquad \text{for}\quad j = \jbar,
\]
%
%
%
we get:
\begin{align*}
\leq& 2\Big(0.64\cdot0.76\log\disc_\K + 0.64\cdot 0.76 n_\K\log\Big(\frac{t+1}{2\pi}\Big)
     + 0.64\cdot 0.82n_\K + 6.92\cdot 0.76 n_\K + 3.49\cdot 0.76\Big)                     \\
    &+\frac{5.4}{1+t^2}\,(0.96-0.64)\log\disc_\K
\end{align*}
which is
\begin{align}
&\leq \Big(1 + \frac{1.8}{1+t^2}\Big)\log\disc_\K
      + (\log(t+1) + 9.8)n_\K
      + 5.31.
\label{eq:C37}
%
%
\intertext{Suppose $0<t<1$. Then the term for $j=\jbar$ disappears and~\eqref{eq:C36} is}
&\leq \log\disc_\K
      + (\log(t+1) + 9.8)n_\K
      + 5.31.
\label{eq:C38}
\end{align}

Lastly we note that
\[
\sum_{\substack{|\gamma-t|\leq1}}\Big|\frac{1}{2+it-\rho}\Big| \leq \frac{2}{3}n_\K(t;1)
\]
which can be bounded with Lemma~\ref{lem:C3}. The claim follows putting all together in~\eqref{eq:C35}
and using~\eqref{eq:C37} and~\eqref{eq:C28a} for $t>1$, and~\eqref{eq:C38} and~\eqref{eq:C28b} for
$0<t<1$. The proof concludes by noticing that the first bound is worst than the second in $0<t<1$ and that
therefore its range can be extended to $t>0$.
%
%
\end{proof}

\begin{lemma}[Bound of $IV$]\label{lem:C11}
Assume GRH. Let $\sigma \in(\frac{1}{2},1)$, $|t|\geq 10$ and $\delta \in (0,1)$. Then
\[
|IV| \leq  \big(3.11\log\disc_\K + (3.11\log(|t|) + 35)n_\K + 20\big)\delta^{\sigma-1/4}.
\]
\end{lemma}
\begin{proof}
By Lemmas~\ref{lem:C2} and~\ref{lem:C10} we get
\begin{align*}
|IVb|
\leq& \frac{\delta^{\sigma-1/4}}{2\pi}
     \Big((4.73\cdot 2.18 + 0.013\cdot 3.2)\log\disc_\K
          +4.73\cdot 11.7 + 0.007\cdot 7                             \\
     &
          +(4.73\cdot 2.18\log(|t|+1)
            + 4.73 \cdot 21.6
            + 0.171\cdot 10
           )n_\K
     \Big).
\end{align*}
We get the claim adding $|IVa|$ as estimated in Lemma~\ref{lem:C9}.
\end{proof}

We are finally able to prove the bound of $\frac{\zeta'_\K}{\zeta_\K}(s)$ in the critical strip.
\begin{lemma}\label{lem:C12}
Assume GRH. Let $\sigma\in(\frac{1}{2},1)$ and $|t|\geq 10$. Then
\[
\Big|\frac{\zeta'_\K}{\zeta_\K}(s)\Big|
\leq \Big(\frac{n_\K}{1-\sigma}
           + \frac{\log(\frac{1}{2\sigma-1})}{\pi}
           + \frac{0.64}{2\sigma-1}
           + 1.37
     \Big)Q^{2-2\sigma}\\
    + \Big(\frac{0.07}{2\sigma-1}
           + 4\Big)n_\K
\]
%
with $Q:= \log\disc_\K + (\log|t| + 20)n_\K + 11$.
\end{lemma}

\begin{proof}
From~\eqref{eq:C30} and Lemmas~\ref{lem:C6}, \ref{lem:C7}, \ref{lem:C8} and~\ref{lem:C11} we get
\begin{align*}
\Big|\frac{\zeta'_\K}{\zeta_\K}(s)\Big|
\leq{}& \delta^{\sigma-\tfrac{1}{2}}\Big(\Big(\frac{\log(\frac{1}{2\sigma-1})}{\pi} + \frac{0.64}{2\sigma-1} + 0.82\Big) W_\K(t)
                                      + \Big(\frac{13.9}{2\sigma-1} + 1.6\Big) n_\K + \frac{6.9}{2\sigma-1} + 0.8
                                 \Big)                                                 \\
      & + \sqrt{2\pi} e^{-\frac{\pi}{2}|t|}\delta^{\sigma-1}
        + \big(3.11\log\disc_\K + (3.11\log |t| + 35)n_\K + 20\big)\delta^{\sigma-1/4} \\
      & + \Big(\frac{\delta^{\sigma-1}}{1-\sigma}
             + \frac{0.07}{2\sigma-1}
             + 4\Big)n_\K
\end{align*}
which we simplify to
\begin{align*}
\Big|\frac{\zeta'_\K}{\zeta_\K}(s)\Big|
\leq{}& \Big(\frac{\log(\frac{1}{2\sigma-1})}{\pi} + \frac{0.64}{2\sigma-1} + 0.82\Big)Q\delta^{\sigma-\tfrac{1}{2}}
       + \sqrt{2\pi} e^{-\frac{\pi}{2}|t|}\delta^{\sigma-1}
       + 3.11 Q\delta^{\sigma-1/4} \\
      &+ \Big(\frac{\delta^{\sigma-1}}{1-\sigma}
              + \frac{0.07}{2\sigma-1}
              + 4\Big)n_\K
\end{align*}
where $Q:= \log\disc_\K + (\log |t| + 20)n_\K + 11$ (thus $W_\K \leq Q -21.8n_\K -11$).\\
We get the claim by setting $\delta := Q^{-2}$ and with some minor simplifications which come from the
assumption $|t|\geq 10$ and the lower bound $Q\geq 33$.
\end{proof}

Finally, the inequality $f_\K(s)\leq \tilde{f}_\K(s)$ with the $\tilde{f}_\K(s)$ given in
Theorem~\ref{th:C1} follows from plugging the estimates of Lemmas~\ref{lem:C1} and~\ref{lem:C12}
in~\eqref{eq:C12}, and simplifying the resulting inequality using the bound $n_\K
\log\big(\frac{\sqrt{\sigma^2+t^2}}{2\pi}\big) + \frac{2\sigma}{\sigma^2+t^2} +
\frac{2\sigma-2}{(\sigma-1)^2+t^2} \leq n_\K\log t$ which holds when $\sigma \in (\tfrac{1}{2},1)$ and
$t\geq 1$.
%

\appendix
\section*{Appendix}
\setcounter{equation}{0}
\renewcommand{\theequation}{A.\arabic{equation}}
\subsection*{Proof of Lemma~\ref{lem:C1}}
Using the explicit formula~\cite[Th.~1.4.2]{AnAsRoy} for $\log\Gamma(s)$ coming from the Euler--Maclaurin
summation formula one gets:
\begin{align*}
\frac{\Gamma'}{\Gamma}(s) - \log\Big(s-\frac{1}{2}\Big)
&= -\log\Big(1-\frac{1}{2s}\Big) -\frac{1}{2s} - \frac{1}{12s^2} + \frac{1}{120s^4} + \int_{0}^{+\infty} \frac{B_4(\fracpart{u})}{(s+u)^5}\d u\\
&= \frac{1}{24s^2} + \sum_{n=3}^{+\infty} \frac{1}{n2^ns^n} + \frac{1}{120s^4} + \int_{0}^{+\infty} \frac{B_4(\fracpart{u})}{(s+u)^5}\d u
\end{align*}
where $B_4(x)= x^4-2x^3+x^2-\frac{1}{30}$. Thus, if $t$ is positive we get
\begin{align*}
\Ree\Big(\frac{\Gamma'}{\Gamma}&(s)-\log\Big(s - \frac{1}{2}\Big)\Big)
\leq  \frac{1}{24}\Ree\Big(\frac{1}{s^2}\Big) + \sum_{n=3}^{+\infty} \frac{1}{n2^n|s|^n} + \frac{1}{120|s|^4} + \int_{0}^{+\infty} \frac{|B_4(\fracpart{u})|}{|s+u|^5}\d u\\
\leq& \frac{1}{24}\Ree\Big(\frac{1}{s^2}\Big) + \frac{1}{12|s|^2(2|s|-1)} + \frac{1}{120|s|^4} + \frac{1}{30}\int_{0}^{+\infty} \frac{\d u}{(u^2+t^2)^{5/2}}\\
\leq& \frac{1}{24}\Ree\Big(\frac{1}{s^2}\Big) + \frac{1}{12|s|^2(2t-1)} + \frac{1}{120|s|^4} + \frac{1}{45t^4}.
\end{align*}
This is a rational function of $\sigma$ and $t$, and with elementary arguments one proves that it is
negative for $\sigma\geq 0$ and $t\geq \sigma+2$.

\subsection*{Proof of Lemma~\ref{lem:C2}}
The map $u\to \int_\R|\Gamma(u+iv)|\d v$ is log-convex for $u\in(-1,0)$ as a consequence of a general
inequality of Hardy, Ingham and P\'{o}lya (see~\cite[Ch.~11, Prop.~4]{NarasimhanNievergelt}), and the uniform
exponential decay of $\Gamma(u+iv)$ for $v\to\infty$. Hence,
\[
\int_\R|\Gamma(u+iy)|\d y \leq \max\Big\{\int_\R\Big|\Gamma\Big(-\frac{3}{4}+iy\Big)\Big|\d y,
                                         \int_\R\Big|\Gamma\Big(-\frac{1}{4}+iy\Big)\Big|\d y
                                   \Big\}
\]
for any $u\in[-3/4,-1/4]$. The last two integrals are bounded respectively by $4.43$ and $4.73$ (an
effective version of the Stirling bound as~\cite[Cor.~1.4.4]{AnAsRoy} may be used to prove that the
contribution of the range $y\in \R\backslash [-10,10]$ to the integral is smaller than $10^{-6}$, and
then, using the monotonicity of $y\mapsto|\Gamma(u+iy)|$, a Riemann sum with $10000$ points produces the
result). This proves~\eqref{eq:C25a}.
%
%
\medskip\\
Without loss of generality we can assume $t>10$. By~\eqref{eq:C25a}, in order to prove~\eqref{eq:C25b} it is
sufficient to show that
\begin{equation}\label{eq:C39}
\int_\R |\Gamma(u+i(t-y))|\log\Big(\frac{1+|y|}{1+t}\Big)\d y
\end{equation}
is negative. Let
\begin{align*}
F_1(u,v) &:= -\int_{v}^{+\infty} |\Gamma(u+iw)|\d w,\\
F_2(u,v) &:= -\int_{v}^{+\infty} F_1(u,w)\d w
           =  \int_{v}^{+\infty} (w-v)|\Gamma(u+iw)|\d w,
\end{align*}
so that $\partial_v F_1(u,v) = |\Gamma(u+iv)|$ and $\partial_v F_2(u,v) = F_1(u,v)$.
We split~\eqref{eq:C39} into three ranges $(-\infty,0]\cup[0,t]\cup[t,+\infty)$, getting
\begin{multline*}
\int_\R
|\Gamma(u+i(t-y))|\log\Big(\frac{1+|y|}{1+t}\Big)\d y
   = \int_{-\infty}^0|\Gamma(u+i(t-y))|\log\Big(\frac{1-y}{1+t}\Big)\d y\\
   +  \int_0^t |\Gamma(u+i(t-y))|\log\Big(\frac{1+y}{1+t}\Big)\d y
   +  \int_t^{+\infty}|\Gamma(u+i(y-t))|\log\Big(\frac{1+y}{1+t}\Big)\d y
\end{multline*}
where in the last term we have used the equality $|\Gamma(u-iw)|=|\Gamma(u+iw)|$ to ensure the positivity
of the imaginary part of the argument of the gamma function. Then an integration by parts shows that it is
\begin{align*}
  &= - \int_{-\infty}^0\frac{F_1(u,t-y)}{1-y}\d y
     + \int_0^t\frac{F_1(u,t-y)}{1+y}\d y
     - \int_t^{+\infty}\frac{F_1(u,y-t)}{1+y}\d y.
\intertext{A second integration by parts produces}
%
%
  &=  2F_2(u,t)
   -  \int_{-\infty}^0\frac{F_2(u,t-y)}{(1-y)^2}\d y
   -  \int_0^t\frac{F_2(u,t-y)}{(1+y)^2}\d y
   -  \int_t^{+\infty}\frac{F_2(u,y-t)}{(1+y)^2}\d y.
\intertext{Function $F_2$ being positive, this is}
  &\leq 2F_2(u,t) - \frac{1}{(1+t)^2}\int_0^t F_2(u,t-y)\d y
   =    2F_2(u,t) - \frac{1}{(1+t)^2}\int_0^t F_2(u,y)\d y
\end{align*}
which is negative if and only if
\begin{equation}\label{eq:C40}
2(1+t)^2F_2(u,t) \leq \int_0^t F_2(u,y)\d y.
\end{equation}
This inequality holds when $t$ is large enough because the left-hand side decreases to $0$ as a function
of $t$. In order to prove that this happens already for $t\geq 10$ we use an effective version of the
Stirling bound~(see~\cite[Cor.~1.4.4]{AnAsRoy}) giving, when $u<0$ and $v>0$,
\[
|\Gamma(u+iv)|
= \sqrt{2\pi} |u+iv|^{u-1/2} e^{- \frac{\pi}{2} v} e^{-u + v\arctan(u/v)} e^{R(u,v)}
\]
with $|R(u,v)| \leq \frac{1}{8v}(\frac{\pi}{2}-\arctan(u/v))$.
Thus, if furthermore $u\in[-3/4,-1/4]$ and $v\geq 10$, we get
\[
|F_2(u+iv)|
\leq  \frac{4}{\pi^2}\sqrt{2\pi} e^{3/4}v^{-3/4}  e^{\frac{\pi}{16v}} e^{- \frac{\pi}{2} v}.
\]
We get that $\forall v\geq 10$ the left-hand side of~\eqref{eq:C40} is lower than $2\cdot 10^{-5}$.
On the contrary,
\begin{align*}
\int_0^t F_2(u,y)\d y
&=     \frac{1}{2}\int_0^t w^2 |\Gamma(u+iw)|\d w
\geq \frac{1}{2}\int_0^{1} w^2 |\Gamma(u+iw)|\d w\\
&\geq \frac{1}{6}\min_{\substack{u\in[-3/4,-1/4]\\ w\in\times[0,1]}} |\Gamma(u+iw)|.
\end{align*}
The maximum modulus principle for holomorphic functions (applied to $1/\Gamma(z)$) shows that the minimum
is reached at the boundary of the region $[-3/4,-1/4]\times[0,1]$, and it is easy to verify that here
$|\Gamma(z)|$ is always larger than $0.4$ thus the right-hand side of~\eqref{eq:C40} is larger than
$0.06$.
%
%
\medskip\\
For~\eqref{eq:C25c} we note that $|\Gamma(u+iv)| \leq 5.3 e^{-\pi |v|/2}$ for every $u\in[-3/4,-1/4]$ and
every $v$ (this can be proved as Lemma~\ref{lem:C7}),
thus it is sufficient to bound
\[
F(t) := \int_\R \frac{e^{-\frac{\pi}{2}|t-y|}}{|1+4iy|}\d y.
\]
It is the unique bounded solution of the differential equation $F''(t)-\frac{\pi^2}{4} F(t) =
-\frac{\pi}{|1+4it|}$. A numerical check shows that $F(10) = 0.032\ldots$ is larger than
$\tfrac{4/\pi}{|1+40i|} = 0.031\ldots$.
%
%
Suppose that there exists $t_0>10$ with $F(t_0)> F(10)$. Then there exists also a point $t_1>t_0$ with
$F(t_1)= \max_{t\in[10,+\infty)}F(t)$, because $F(+\infty) = 0$. Then $F''(t_1)\leq 0$, because $F$ is a
$C^2$ map in $(10,+\infty)$. Thus from the differential equation we get
\[
\frac{\pi^2}{4} F(t_1) \leq \frac{\pi}{|1+4it_1|} < \frac{\pi}{|1+40i|} <  \frac{\pi^2}{4} F(10)
\]
which violates the definition of $t_1$. This proves that $F(10)= \max_{t\in[10,+\infty)}F(t)$, so that
\[
\int_\R \frac{|\Gamma(u+i(t-y))|}{|1+4iy|}\d y
\leq 5.3 F(t)
\leq 5.3 F(10)
\leq 0.171.
\]
\medskip\\
The same argument may be applied to prove~\eqref{eq:C25d}, since $F_\alpha(t) := \int_\R
\frac{e^{-\frac{\pi}{2}|t-y|}}{1+\alpha y^2}\d y$ is the unique solution of the differential equation
$F_\alpha''(t)-\frac{\pi^2}{4} F_\alpha(t) = -\frac{\pi}{1+\alpha y^2}$ which is bounded for
$t\to\pm\infty$, and $F_1(10) = 0.0129\ldots > \frac{4/\pi}{1+10^2} = 0.0126\ldots$, and $F_2(10) =
0.0065\ldots > \frac{4/\pi}{1+2\cdot10^2} = 0.0063\ldots$.
%
%


\end{document}